\documentclass[12pt]{elsarticle}

\usepackage{latexsym}
\usepackage{amssymb,amsthm,amsmath}
\usepackage{graphicx}
\usepackage{makeidx}
\usepackage{comment}
\usepackage{color}
\usepackage{lineno}
\usepackage{url}

\newcommand{\ff}[1]{{\mathbb F}_{#1}}

\newcommand{\ffx}[1]{\ff{#1}[x]}

\newtheorem{theorem}{Theorem}[section]
\newtheorem{proposition}[theorem]{Proposition}
\newtheorem{lemma}[theorem]{Lemma}
\newtheorem{corollary}[theorem]{Corollary}

   {\theoremstyle{definition}
\newtheorem{definition}[theorem]{Definition}

\newtheorem{example}[theorem]{Example}
}
   \theoremstyle{remark}
\newtheorem{remark}[theorem]{Remark}

\newtheorem{open}[theorem]{Open Problem}

\journal{Journal Name}

\begin{document}

\begin{frontmatter}


\title{On construction and (non)existence of $c$-(almost) perfect nonlinear functions}



\author[1]{Daniele Bartoli}
\ead{daniele.bartoli@unipg.it}
\address[1]{Department of Mathematics and Computer Sciences, University of Perugia, Italy}
\author[2]{Marco Calderini}
\ead{marco.calderini@uib.no}
\address[2]{Department of Informatics, University of Bergen, Norway}


\begin{abstract}
Functions with low differential uniformity have relevant applications in cryptography. Recently, functions with low  $c$-differential uniformity attracted lots of attention. In particular, so-called APcN and PcN functions (generalization of APN and PN functions) have been investigated. Here, we provide a characterization of such functions via quadratic polynomials as well as non-existence results.
\end{abstract}


\end{frontmatter}


\section{Introduction}
{\em Perfect nonlinear} (PN) and {\em almost perfect nonlinear} (APN) functions and in general functions with low differential uniformity over finite fields have been widely investigated due to their applications in cryptography. Indeed, differential cryptanalysis \cite{BihamShamir,BihamShamir2} is an important cryptanalytic approach targeting symmetric-key primitives. In order to be resistant against such types of attacks, cryptographic functions used in the substitution box (S-box)  in the cipher are required to have a differential uniformity as low as possible (see \cite{Carlet1} for a survey on differential uniformity of vectorial Boolean functions). In \cite{Borisov}, the authors introduce a different type of differential, useful for ciphers that utilize modular multiplication as a primitive operation.  Consequently, a new concept called multiplicative differential (and the corresponding $c$-differential uniformity) has been introduced \cite{EFRST2020}. 

\begin{definition}\cite[Definition 1]{EFRST2020}
Given a $p$-ary $(n,m)$-function $f : \mathbb{F}_{p^n} \to  \mathbb{F}_{p^m}$, and $c\in  \mathbb{F}_{p^m}$, the (multiplicative) $c$-derivative of $f$ with respect to $a\in \mathbb{F}_{p^n}$ is the function
$$_cD_af(x)=f(x+a)-cf(x), \quad \ \forall x \in \mathbb{F}_{p^n}.$$
\end{definition}

For an $(n,n)$-function $f$, and $a,b\in \mathbb{F}_{p^n}$,  {let} 
$$_c\Delta_f(a,b) := {|}\{x \in \mathbb{F}_{p^n} \ : \ f(x+a)-cf(x)=b\}{|},$$
and 
$$_c\Delta_f := \max \{_c\Delta_f(a,b) \ :\ a,b \in \mathbb{F}_{p^n}, (a,c)\neq (0,1)\},$$
{where $|S|$ is the cardinality of the set $S$.} 
The quantity $_c\Delta_f$ is called $c$-differential uniformity of $f$. Note that for $c=1$, the above definitions coincide with the usual derivative of $f$ and its differential uniformity. 

If $_c\Delta_f\leq \delta \in \mathbb{N}$, we say that $f$ is  differentially $(c,\delta)$-uniform. In the special cases $\delta=1$ and $\delta=2$, such functions are also called PcN and APcN functions. It is worth noting that PcN functions (namely $\beta$-planar functions) have been investigated and partially classified in \cite{BT2019}. 

Clearly, the case $c=1$ (APN and PN functions) has been widely investigated in the literature; see \cite{BCCCV20,BCL08,Dobbertin,Dobbertin2,Dobbertin3, Gold, Janwa,Kasami,Nyberg} and \cite{CM1997,DO1968,DY2006,Dobbertin4,Helleseth1,Helleseth2,L2012,ZW2011} for known APN and PN functions. PN functions are also called \emph{planar}. APN and PN functions are of central interest in design theory, coding theory, and cryptography.

Very recently, power functions with low $c$-differential uniformity, and the $c$-differential uniformity of some known APN functions in odd characteristic have been studied in \cite{RS20}. Also in \cite{HPRS20}, the authors focus on monomial functions and study their $c$-differential uniformity for $c=-1$.

In this paper, we further investigate the construction and existence of some APcN and PcN functions. First, in Section \ref{sec:pre}, we collect some preliminary results and definitions that we will use in the rest of the paper. In Section \ref{sec:con}, we first give a characterization of APcN and PcN quadratic functions, which, in particular, gives us a correspondence between planar DO polynomials and APcN maps. Then, we show that, using the AGW criterion \cite{AGW} and its generalization \cite{Mesn}, it is possible to construct several classes of APcN and PcN functions. 
In the last section, we give some nonexistence results for some exceptional monomial APcN and PcN functions using connections with algebraic curves and Galois theory tools.


\section{Preliminaries}\label{sec:pre}

Let $q=p^n$ be a fixed prime power. We denote by $\mathbb{F}_q$ and $\overline{\mathbb{F}_q}$ the field with $q$ elements and its algebraic closure. { The multiplicative group of $\mathbb{F}_q$ will be denoted by $\mathbb{F}_q^*=\mathbb{F}_q\setminus\{0\}$.}  In the following we will focus on functions defined from $\ff{q}$ to itself, i.e. $p$-ary $(n,n)$-functions.
Any function $f:\ff{q}\to\ff{q}$ can be
represented uniquely by an element of the polynomial ring $\ffx{q}$ of degree
less than $q$.  

For $f\in\ffx{q}$:
\begin{itemize}
\item $f$ is {\em linear} if $F(x)=\sum_i a_i x^{p^i}$ (also known as linearised polynomials). 

\item $f$ is {\em affine} if it differs from a linear polynomial by a
constant. 
\item $f$ is a {\em Dembowski-Ostrom (DO)} polynomial if
$F(x)=2\sum_{0\le i\le j<n} a_{ij} x^{p^i+p^j}$, with $i<j$ if $p=2$.

\item $f$ is {\em quadratic} if it differs from a DO polynomial by an affine
polynomial. 
\end{itemize}
{The \emph{trace function} from $\mathbb{F}_{q^n}$ to $\mathbb{F}_q$ is given by the linear polynomial
$$
Tr_q^{q^n}(x)=\sum_{i=0}^{n-1}x^{q^i}.
$$}
A polynomial $f$ is a {\em permutation polynomial (PP)} over $\ff{q}$, if
$x\mapsto f(x)$ is a bijection from $\ff{q}$ to itself, and it is a {\em complete permutation polynomial (CPP)} over $\ff{q}$, if both $f(x)$ and
$f(x)+x$ are PPs.


The AGW criterion, introduced in \cite{AGW}, is a useful method in the construction of PPs and CPPs; see for instance \cite{LWWZ2014,XFZ2019,YD2011,YD2014}.
The AGW criterion, in the additive case, is given by the following proposition.

\begin{proposition}[Proposition 5.4 \cite{AGW}]\label{prop:agwc}
Let $p$ be a prime and $q=p^m$ for some integer $m>0$. Let $\phi(x)$ and $\psi(x)$ be two $\mathbb{F}_q$-linear polynomials over $\mathbb{F}_q$ seen as endomorphisms of $\mathbb{F}_{q^n}$, and let $g\in\mathbb{F}_{q^n}[x]$ and  $h\in\mathbb{F}_{q^n}[x]$ such that $h(\psi(\mathbb{F}_{q^n}))\subseteq \mathbb{F}_q^*$.Then
$$f(x)=h\circ \psi(x) \phi(x)+g\circ\psi(x)$$
is a permutation polynomial of $\mathbb{F}_{q^n}$ if and only if the following two conditions hold:
\begin{itemize}
\item[(i)] $\ker(\phi)\cap\ker(\psi)=\{0\}$;
\item[(ii)] $h(x)\phi(x)+\psi(g(x))$ permutes $\psi(\mathbb{F}_{q^n})$.
\end{itemize}
\end{proposition}

As immediate consequence, in Theorem~5.10 in \cite{AGW} the authors provided the following general framework of PPs.
\begin{theorem}[\cite{AGW}]\label{prop:agw}
Let $p$ be a prime and $q=p^m$ for some integer $m>0$. Let $\phi(x)$ be an $\mathbb{F}_q$-linear polynomial over $\mathbb{F}_q$ seen as endomorphism of $\mathbb{F}_{q^n}$, and let $g\in\mathbb{F}_{q^n}[x]$ and  $h\in\mathbb{F}_{q^n}[x]$ such that $h(x^q-x)\subseteq \mathbb{F}_q^*$.Then
$$f_1(x)=h(x^q-x)\phi(x)+Tr_q^{q^n}(g(x^q-x))$$ and $$f_2(x)=h(x^q-x)\phi(x)+g(x^q-x)^{(q^n-1)/(q-1)}$$
are permutation polynomials of $\mathbb{F}_{q^n}$ if and only if $\ker(\phi)\cap \mathbb{F}_q=\{0\}$ and $h(x)\phi(x)$ permutes $J=\{x^q-x: x \in \mathbb{F}_{q^n} \}$.
\end{theorem}

In \cite{Mesn}, Mesnager and Qu extended the AGW criterion for constructing $2$-to-$1$ map. If $q$ is even, a 2-to-1 map over $\ff{q}$ is a function such that any $b\in\ff{q}$ has either 2 or 0 preimages.
If $q$ is odd, for all but one $b\in\ff{q}$, it has either 2 or 0 preimages, and the
exception element has exactly one preimage.

For $q=2^m$, using $\phi$ a 2-to-1 map over $\mathbb{F}_q$ and that permutes $J=\{x^q+x\,:\,x\in \mathbb{F}_{q^n}\}$ it is possible to construct $2$-to-$1$ maps of same type as in 
Theorem~\ref{prop:agw}.
More specifically, we have the following result.

\begin{theorem}[Theorem~15 \cite{Mesn}]\label{prop:agw2a1}
Let $q=2^m$, $\phi(x)$ be an $\mathbb{F}_q$-linear polynomial seen as an endomorphism of $\mathbb{F}_{q^n}$. Let $g,h\in\mathbb{F}_{q^n}[x]$ be such that $h(x^q+x)\in\mathbb{F}_{q}^{*}$ for any $x\in \mathbb{F}_{q^n}$. Assume
$$f_1(x)=h(x^q+x)\phi(x)+Tr_q^{q^n}(g(x^q+x))$${ and }$$f_2(x)=h(x^q+x)\phi(x)+g(x^q+x)^{(q^n-1)/(q-1)}.$$
If $\phi$ is 2-to-1 over $\mathbb{F}_q$ and $h(x)\phi(x)$ permutes $J=\{x^q+x\,:\, x \in\mathbb{F}_{q^n}\}$, then both $f_1$ and $f_2$ are 2-to-1 over $\mathbb{F}_{q^n}$.
\end{theorem}

In the second part of this work, Section \ref{S:1}, we deal with exceptional power APcN and PcN maps.

\begin{definition}
Let $c\in\ff{q}$ be fixed. Let $f(x) \in \ffx{q}$ be a APcN (PcN) function over $\ff{q^r}$ for infinitely many $r$. Then, $f$ is said exceptional APcN (PcN).
\end{definition}

Results on exceptional APN and PN functions can be found in \cite{Survey,State} and the references therein.


We use Galois theory tools to provide non-existence results for  APcN and PcN monomials. We recall here the Galois theoretical part of our approach which deals with totally split places. This method was successfully used also in \cite{BM2020,ferraguti2018full,micheli2019constructions,micheli2019selection}.

We will make use of the following results.

\begin{theorem}\label{thm:hemultsn}\cite[Theorem~3.9]{Helmut}
Let $r$ be a prime and $G$ be a primitive group
of degree $n = s + k$ with $k \geq 3$. If $G$ contains an element of
degree and order $s$ (i.e. an $s$-cycle), then $G$ is either alternating or symmetric.
\end{theorem}

The proof of the following result can be found  in \cite{guralnick2007exceptional}.

\begin{lemma}\label{orbits}
Let $L:K$ be a finite separable extension of function fields, let $M$ be its Galois
 closure and $G:= \textrm{Gal}(M:K)$ be its Galois group. 
 Let $P$ be a place of $K$ and $\mathcal Q$ be the set of places of $L$ lying above $P$.
Let $R$ be a place of $M$ lying above $P$. Then we have the following:
\begin{enumerate}
\item There is a natural bijection between $\mathcal Q$ and the set of orbits of $H:=\mathrm{Hom}_K(L,M)$ under the action of the decomposition group $D(R|P)=\{g\in G\,|\, g(R)=R\}$.
\item  Let $Q\in \mathcal Q$ and let $H_Q$ be the orbit of $D(R|P)$ corresponding to $Q$. Then $|H_Q|=e(Q|P)f(Q|P)$ where $e(Q|P)$ and $f(Q|P)$ are ramification index and relative degree, respectively. 

\item The orbit $H_Q$ partitions further under the action of the inertia group $I(R|P)$ into $f(Q|P)$ orbits of size $e(Q|P)$. 
\end{enumerate}
\end{lemma}

The following can also be deduced by \cite{kosters2014short}; its proof can be found in \cite{BM2020}.
\begin{theorem}\label{thm:existence_tot_split}
Let $p$ be a prime number, $m$ a positive integer, and $q=p^m$.
Let $L:F$ be a separable extension of global function fields over $\mathbb F_q$ of degree $n$,  $M$ be the Galois closure of $L:F$, and suppose that the field of constants of $M$ is $\mathbb F_q$.
There exists an explicit constant $C\in \mathbb R^+$ depending only on the genus of $M$ and the degree of $L:F$ such that if $q>C$ then $L:F$ has a totally split place.
\end{theorem}

\section{A characterization of APcN and  PcN functions}\label{sec:con}

It is well-known that a DO polynomial is planar if and only if it is 2-to-1 (see \cite[Theorem 3]{CM2011}). 
The following result gives a characterization of APcN and PcN quadratic polynomials for $c\in\mathbb{F}_p\setminus\{1\}$. 

Let $f:\mathbb{F}_q\to\mathbb{F}_q$. We say that $f$ is at most 2-to-1 function if for any $b\in\mathbb{F}_q$ we have $|f^{-1}(b)|\le 2$.
\begin{theorem}\label{th:DO}
Let $p$ be a prime. Let $f$ be a quadratic polynomial over $\mathbb{F}_{p^m}$ for some integer $m$. Then, for any $c\in \mathbb{F}_{p}\setminus\{1\}$ we have the following.
\begin{itemize}
\item[(i)] $f$ is at most 2-to-1 if and only if $f$ is APcN. Moreover, if $f$ is a DO polynomial, then $f$ is APcN if and only if $f$ is planar.
\item[(ii)] $f$ is a PP if and only if $f$ is PcN.
\end{itemize}
\end{theorem}
\begin{proof}
{\bf (i)} Let $f$ be a quadratic polynomial, that is $f(x)=\sum_{i,j}a_{i,j}x^{p^i+p^j}+\sum_{i}b_{i}x^{p^i}$.
We can note that for any $\gamma$ we have $$f(x+\gamma)=f(x)+f(\gamma)+\sum_{i,j}a_{i,j}(x^{p^i}\gamma^{p^j}+x^{p^j}\gamma^{p^i}).$$

Let $c\in \mathbb{F}_{p}\setminus\{1\}$. Then 
\begin{equation}\label{eq:cder}
\begin{aligned}
    f(x+\gamma)-cf(x)=&(1-c)\left(f(x)+\sum_{i,j}a_{i,j}\left[x^{p^i}\left(\frac{\gamma}{1-c}\right)^{p^j}+x^{p^j}\left(\frac{\gamma}{1-c}\right)^{p^i}\right]\right.\\
    &\left.+f\left(\frac{\gamma}{1-c}\right)-f\left(\frac{\gamma}{1-c}\right)\right)+f(\gamma)\\
    =&(1-c)f\left(x+\frac{\gamma}{1-c}\right)+f(\gamma)-(1-c)f\left(\frac{\gamma}{1-c}\right).
\end{aligned}
\end{equation}
Thus,  since $f$ is at most 2-to-1 so it is $f(x+\gamma)-cf(x)$, which implies that $f$ is APcN, and vice versa.

If $f$ is a DO polynomial then $f(x)=f(-x)$. Therefore, the fact that $f$ is at most 2-to-1 would imply that $f$ is 2-to-1, and so it is planar.

\noindent {\bf (ii)} This follows directly from \eqref{eq:cder}.
\end{proof}

{\begin{corollary}
Let $p$ be a prime, and $f$ be a DO polynomial over $\mathbb{F}_{p^m}$, with $m$ a positive integer. Then, $f$ is exceptional planar if and only if $f$ is exceptional APcN for any $c\in\mathbb{F}_p\setminus\{1\}$.
\end{corollary}}

\begin{remark}
Let $q=p^h$. If in Theorem \ref{th:DO} the quadratic function $f$ is of type 
$$
f(x)=\sum_{i,j}a_{i,j}x^{q^i+q^j}+\sum_{i}b_{i}x^{p^i},
$$
then the results above can be extended to any $c\in\mathbb{F}_q\setminus\{1\}$. 
\end{remark}

Up to now, all known planar functions are DO polynomials, but the case of $x^{\frac{3^k+1}{2}}$ defined over $\ff{3^n}$ with $k$ odd and $\gcd(k,n)=1$. From Theorem~\ref{th:DO}, we have that these known planar functions are also APcN. Moreover, in \cite{RS20} it has been proved that the planar function $x^{\frac{3^k+1}{2}}$ is APcN for $c=-1$. 

The result (i) of Theorem~\ref{th:DO} cannot be extended to a general planar quadratic function. Indeed, the planarity of a function $f$ is invariant by adding a linear (affine) polynomial to $f$, while the $c$-differential uniformity is not. So, if we consider a planar DO polynomial, adding a linear function we could obtain a function which is no more 2-to-1 and thus which is no APcN.
\begin{example}
The function $x^2+x^3$ is planar over $\ff{3^2}$ but it is not APcN for any $c\ne 1$.
\end{example}

\begin{remark}
In \cite{SGGRT2020}, the authors introduced and studied c-differential bent functions. In their work, they also relaxed the definition of perfect $c$-nonlinearity excluding the case of the derivative in the zero direction. In particular, they defined PcN function any $f$ such that $f(x+\gamma)-cf(x)$ is a permutation for any $\gamma\in\mathbb{F}_{q}^*$, and \emph{strictly PcN} if in addition  $f$ is a permutation.

For $p=2$, even if we exclude the derivative in the zero direction, a PcN function has to be a PP. Indeed, let $f$ be PcN and suppose that there exist $x_1$ and $x_2=x_1+\gamma$ such that $f(x_1)=f(x_1+\gamma)$. Since $f$ is PcN,
$$
f(x+\gamma)+cf(x)=(c+1)f(x)+f(x+\gamma)+f(x)
$$
is a PP. But $$
f(x_1+\gamma)+cf(x_1)=(c+1)f(x_1)=(c+1)f(x_2)=f(x_2+\gamma)+cf(x_2),
$$
which is a contradiction.

It would be interesting to understand if this is the case also for $p>2$.
\end{remark}


\subsection{Some PcN and APcN polynomials from the AGW criterion}

In the following we will show that from the AGW criterion and its generalization  \cite{Mesn} (for the case $p=2$) we can obtain PcN and APcN functions.




Theorem~\ref{prop:agw} gives us the possibility of constructing PPs of the form
$$f_1(x)=h(x^q-x)\phi(x)+Tr_q^{q^n}(g(x^q-x))$$ and $$f_2(x)=h(x^q-x)\phi(x)+g(x^q-x)^{(q^n-1)/(q-1)},$$
where $g$ can be any polynomial over $\ff{q^n}$.
This is implied by the fact that $x^q-x$ annihilates both $Tr_q^{q^n}(g(x))$ and $g(x)^{(q^n-1)/(q-1)}$ for any $x$.  We can immediately construct some PcN polynomials.

\begin{theorem}\label{th:pcn1}
Let $f_1$ and $f_2$ be PPs as in Theorem~\ref{prop:agw} with $h\equiv b\in \mathbb{F}_q^*$. Then $f_1$ and $f_2$ are PcN for any $c\in\mathbb{F}_q\setminus\{1\}$.
\end{theorem}
\begin{proof}
Let $c\in\mathbb{F}_q\setminus\{0,1\}$. Consider for instance the permutation $f_1$. Then, $f_1$ is PcN if and only if 
$$
f_1(x+\gamma)-cf_1(x)=b(1-c)\phi(x)+Tr_q^{q^n}(g(x^q-x+\gamma^q-\gamma))-cTr_q^{q^n}(g(x^q-x)) +b\phi(\gamma)
$$ 
is a PP for any $\gamma$. Denoting by $\psi(x)=x^q-x$, {and by $g'(x)=g(x+\gamma^q-\gamma)$}, from the AGW criterion (Proposition~\ref{prop:agwc}) we have that this is a PP if and only if
$$
b(1-c)\phi(x)+\psi(Tr_q^{q^n}(g'(x))-cTr_q^{q^n}(g(x)))
$$
permutes $J=\{x^q-x: x \in \mathbb{F}_{q^n} \}$. Now, $\psi(Tr_q^{q^n}(g(x))-cTr_q^{q^n}(g(x)))=0$ and thus $b(1-c)\phi(x)$ permutes $J$ since $f_1$ is a PP. The same holds for $f_2$.
\end{proof}

Another type of PPs, which are also PcN, can be constructed in the following way.

\begin{theorem}\label{th:quad}
Let $p$ be a prime and $q=p^m$ for some integer $m>0$. Let $g(x)\in\mathbb{F}_{q^2}[x]$ be any polynomial such that $g(J)\subseteq J$ where $J=\{x^q-x: x \in \mathbb{F}_{q^{2}} \}$ and $\phi(x)$ be an $\mathbb{F}_q$-linear polynomial over $\mathbb{F}_q$. Let {$s>0$ be an even integer.}
Then, for any $b\in\mathbb{F}_q^*$
$$
f(x)=b\phi(x)+(g(x^q-x))^s
$$
is a PP if and only if $\phi(x)$ induces a permutation over $J$.
\end{theorem}
\begin{proof}
From the AGW criterion (Proposition~\ref{prop:agwc}) we have that $f$ is a PP if and only if
$$
(g(x))^{qs}-(g(x))^s+b\phi(x)
$$
permutes $J$.

Note that for any $y\in J$ we have $Tr^{q^2}_q(y)=0$ and thus $y^q=-y$. 
Since $s$ {is even},
for any $y\in J$ we have $y^{s} \in\mathbb{F}_{q}$. Indeed, $$y^{sq}=(-y)^{s}=y^{s}.$$
Then, since $g(J)\subseteq J$ we have that $$(g(x))^{qs}-(g(x))^s=0,$$
for any $x\in J$. Thus,
$f$ is a PP if and only if $\phi(x)$ permutes $J$.
\end{proof}

\begin{example}
An easy example of function $g$ such that $g(J)\subseteq J$ is given by $g(x)=x+\delta$ with $\delta\in J$.
\end{example}

Theorem~\ref{th:quad} can be generalized (with a similar proof)  to functions $f$ of type
$$
f(x)=b\phi(x)+\sum_{i}(g_i(x^q-x))^{s_i},
$$
where $s_i$'s {are even},
and $g_i$'s are such that $g_i(J)\subseteq J$.
\begin{corollary}
Let $p$ be a prime and $q=p^m$ for some integer $m>0$. Let $t$ be a positive integer.  Let $g_1,...,g_t\in\mathbb{F}_{q^2}[x]$ be such that $g_i(J)\subseteq J$ for all $1\le i\le t$, where $J=\{x^q-x: x \in \mathbb{F}_{q^{2}} \}$, and $\phi(x)$ an $\mathbb{F}_q$-linear polynomial over $\mathbb{F}_q$. Let $s_1,...,s_t$ {be even integers}.
Then, for any $b\in\mathbb{F}_q^*$
$$
f(x)=b\phi(x)+\sum_{i}(g_i(x^q-x))^{s_i},
$$
is a PP if and only if $\phi(x)$ induces a permutation over $J$.
\end{corollary}

\begin{remark}
Note that the polynomials in Theorem~\ref{prop:agw} and \ref{th:quad}, considering $\phi(x)=x$, are also CPPs when $b\ne 0, -1$.
\end{remark}

As for the case of the functions $f_1$ and $f_2$ of Theorem~\ref{prop:agw}, also the functions satisfying Theorem~\ref{th:quad} are PcN when $c\in\mathbb{F}_q\setminus\{1\}$.
\begin{theorem}\label{th:pcn2}
Let $p$ be a prime and $q=p^m$ for some integer $m>0$. Let $f(x)$ be a PP as in Theorem~\ref{th:quad}. Then $f(x)$ is PcN for any $c\in\mathbb{F}_q\setminus\{1\}$.
\end{theorem}
\begin{proof}
We have that
$$f(x+\gamma)-cf(x)=b(1-c)\phi(x)+(g'(x^q-x))^s-c(g(x^q-x))^s+b\phi(\gamma),$$
where $g'(x)=g(x+\gamma^q-\gamma)$. Note that since $J$ is an $\mathbb{F}_q$-vector space,  $g'(J)\subseteq J$.
Now as in Theorem~\ref{th:quad}, this is a permutation if and only if $\phi(x)$ permutes $J$. This condition is satisfied since $f$ is a PP.
\end{proof}

\begin{remark}
In even characteristic, PN functions (i.e. PcN function with $c=1$) do not exist. As pointed out in \cite{EFRST2020}, PcN functions, for $c\ne 1$, exist also for the case $p=2$. Indeed, trivially, any PP is PcN for $c=0$ and any linear permutation is PcN for any $c\ne 1$.   Theorems \ref{th:pcn1} and \ref{th:pcn2} provide non-trivial PcN functions for $p=2$.
\end{remark}

A similar argument can  be done for the case of APcN maps using the results of \cite{Mesn}. 
As for the PcN case we can obtain APcN maps for any $c\in \mathbb{F}_q\setminus\{1\}$ using functions as in Theorem~\ref{prop:agw2a1}. In particular, for $n$ odd, we can obtain the following APcN maps.
\begin{theorem}\label{th:apcnagw}
Let $n$ and $m$ be two positive integers with $n$ odd. Let $q=2^m$ and $\phi(x)$ be an $\ff{q}$-linear polynomial which is 2-to-1 over $\mathbb{F}_q$ and that permutes $J=\{x^q+x\,:\, x \in\mathbb{F}_{q^n}\}$.
Let $g\in\mathbb{F}_{q^n}[x]$ and $b\in\mathbb{F}_q^*$. Then,
$$f_1(x)=b\phi(x)+Tr_q^{q^n}(g(x^q+x))\text{ and }f_2(x)=b\phi(x)+g(x^q+x)^{(q^n-1)/(q-1)}$$
are APcN functions for any $c\in\mathbb{F}_q\setminus\{1\}$.
\end{theorem}
\begin{proof}
Let us consider $f_1(x)$. For any $\gamma$ we have
$$
\begin{aligned}
f_1(x+\gamma)+cf_1(x)=&b\phi(x)+b\phi(\gamma)+Tr_q^{q^n}(g(x^q+x+\gamma^q+\gamma))\\
&+cb\phi(x)+Tr_q^{q^n}(cg(x^q+x))\\
=& b(c+1)\phi(x)+Tr_q^{q^n}(g'(x^q+x))+b\phi(\gamma),
\end{aligned}
$$
where $g'(x)=g(x+\gamma^q+\gamma)+cg(x)$. 
Then, $f_1(x+\gamma)+cf_1(x)$ is 2-to-1 from Theorem~\ref{prop:agw2a1}.

For $f_2$ the claim follows in a similar way.
\end{proof}

\begin{example}
For constructing APcN functions as in Theorem~\ref{th:apcnagw}, we can consider, for example, the 2-to-1 function $\phi$ over $\ff{q}$ defined by $\phi(x)=x^{2^i}+x$ with $\gcd(i,m)=1$.

Indeed, since $\gcd(i,m)=1$ we have that $\ker(\phi)=\mathbb{F}_2$, implying that $\phi$ is 2-to-1 over $\mathbb{F}_q$. 
Moreover $\phi$ permutes $J$. Suppose that there exist $x_1,x_2\in J$ such that $\phi(x_1)=\phi(x_2)$ then $\phi(x_1+x_2)=0$. Since $J$ is a vector subspace, we have $x_1+x_2\in J\cap \ker(\phi)=\{0\}$, recall that $n$ is odd and $Tr_q^{q^n}(1)=1$.
\end{example}

\begin{remark}
Note that, when $n$ is even, it is not possible to construct  $\phi$ that is a 2-to-1 map over $\mathbb{F}_q$ and permutes $J$ since $\mathbb{F}_q\subseteq J$. Indeed $\mathbb{F}_{q^2}$ is a subfield of $\mathbb{F}_{q^n}$ and, denoting by $\psi(x)=x^q+x$, we have $\psi(\mathbb{F}_{q^2})=\mathbb{F}_q$. 

So, for $n$ even, it is not possible to construct APcN functions as in Theorem~\ref{th:apcnagw}.
\end{remark}

\section{Non-existence results for APcN and PcN monomials}
\label{S:1}

In this section we provide non-existence results for exceptional APcN (and PcN) monomials. In what follows, we will consider exponents  $d$ such that $p\nmid d(d-1)$, and we denote $p^h$ by $q$, for some integer $h$, and by $s$ the smallest positive integer such that $d-1 \mid (p^s-1)$.

Let us consider $f(x)=x^d$ defined over $\ff{q}$. The monomial $f(x)$ is APcN, $c\neq 1$,  if and only if
\begin{equation}\label{Condizione}
\forall a,b \in \mathbb{F}_q \Longrightarrow (x+a)^d-cx^d =b \textrm{ has at most two solutions.}
\end{equation}
For $a=0$, the condition above implies that $x^d$ is at most a 2-to-1 function. That is ${\gcd}(d,q-1)\leq 2$. 

When $a\neq 0$, Condition \eqref{Condizione} can be simplified to 
\begin{equation}\label{Condizione2}
\forall b \in \mathbb{F}_q \Longrightarrow (x+1)^d-cx^d =b \textrm{ has at most two solutions.}
\end{equation}

A standard tool, when dealing with APN or PN functions is to consider the curve  $\mathcal{C}_{f,c}$ of affine equation 
\begin{equation}\label{Eq:C_f}
\mathcal{C}_{f,c} \ : \ \frac{(X+1)^d-(Y+1)^d-c (X^d-Y^d)}{X-Y}=0.
\end{equation}

We refer to \cite{BT2019} for and the references therein for an introduction to basic concepts about curves over finite fields. 

{Note that Condition \eqref{Condizione2} implies the existence of at most $q/2$ values $b_i$ for which $(x+1)^d-cx^d =b_i$ has two solutions. 
Therefore, there are at most  $q/2$ pairs $\{x_i,y_i\}$, $x_i\neq y_i$, $x_i,y_i\in \mathbb{F}_q$, such that  $x_i$ and $y_i$ satisfy $(x_i+1)^d-cx_i^d =b_i=(y_i+1)^d-cy_i^d$. Thus, $\mathcal{C}_{f,c}$ possesses at most $q$ $\mathbb{F}_q$-rational points. If $q$ is large enough with respect to $d$, the existence of more than one absolutely irreducible component of $\mathcal{C}_{f,c}$ defined over $\mathbb{F}_q$ would imply, by Hasse-Weil bound, the existence of roughly $2q$ $\mathbb{F}_q$-rational points, a contradiction.}

First, we will provide sufficient conditions on $c$ and $d$ for which $\mathcal{C}_{f,c}$ is absolutely irreducible. In particular, we provide upper bounds on the number of singular points of $\mathcal{C}_{f,c}$. To this end we will consider, for simplicity, the curve $\mathcal{D}_{f,c} : (X+1)^d-(Y+1)^d-c (X^d-Y^d)=0$. Singular points of  $\mathcal{C}_{f,c}$ are a subset of the singular points of $\mathcal{D}_{f,c}$.

\begin{theorem}\label{Th:NoSingPoints}
{Let $\xi\in \overline{\mathbb{F}_q}$ be a primitive $(d-1)$-root of unity. Suppose that  \begin{equation}\label{Eq:Th}
    \nexists \ \ i,j, k\in\{0,\dots,d-2\},  i\ne 0, \ \textrm{such that} \ \sqrt[d-1]{c}\ne \frac{1-\xi^i}{\xi^k-\xi^j}.
\end{equation} Then, $\mathcal{D}_{f,c}$ contains no singular points off $X=Y$. In particular, this is true if $\sqrt[d-1]{c}\notin \mathbb{F}_{p^s}$.}
\end{theorem}
\proof
Since $p\nmid d$, $\mathcal{D}_{f,c}$ does not possess singular points at infinity. Note that there are no singular points lying on $X=0$ or $Y=0$. 
Affine singular points $(x_0,y_0)$, $x_0\neq y_0$, satisfy 
\begin{equation}\label{Eq:SingularPoints}
    \left\{
    \begin{array}{l}
    (\frac{x_0+1}{x_0})^{d-1}=c\\
    (\frac{y_0+1}{y_0})^{d-1}=c\\
    (\frac{x_0}{y_0})^{d-1}=1\\
    \end{array}
    \right..
\end{equation}
Let $\xi\in \overline{\mathbb{F}_q}$ be a primitive $(d-1)$-root of unity and denote by $c_0=\sqrt[d-1]{c}$.  Therefore, $y_0=\xi^i x_0$, $y_0=1/(c_0\xi^{j}-1)$, $x_0=1/(c_0\xi^{k}-1)$, for some $i,j,k \in \{0,\ldots,d-2\}$ and $i\neq 0$. Each triple $(i,j,k)$ provides a pair $(x_0,y_0)$ satisfying \eqref{Eq:SingularPoints}. Thus,
\begin{equation}\label{Eq:c_0}
    c_0\xi^{k}-1=\xi^i (c_0\xi^{j}-1).
\end{equation}
By our hypothesis $\xi \in \mathbb{F}_{p^s}$. Equation \eqref{Eq:c_0} yields
$$c_0(\xi^{k}-\xi^{i+j})=1-\xi^i.
$$
Since $i\neq 0$, 
{we have} a contradiction. 
So, no pairs $(x_0,y_0)$ satisfy \eqref{Eq:SingularPoints} and there are no singular points. 
\endproof

Note that, under the hypothesis of Theorem~\ref{Th:NoSingPoints}  the number of singular points of $\mathcal{C}_{f,c}$ is at most $d/2$. A deeper analysis shows that 
\begin{eqnarray*}
(X+1+a)^d-(Y+1+a)^d-c((X+a)^d-(Y+a)^d)\\
=d[(a+1)^{d-1}-ca^{d-1}](X-Y)+\binom{d}{2}[(a+1)^{d-2}-ca^{d-2}](X^2-Y^2)+\cdots
\end{eqnarray*}
and therefore points $(a,a)$ are double points of $\mathcal{D}_{f,c}$ and then simple points of $\mathcal{C}_{f,c}$. So, $\mathcal{C}_{f,c}$ possesses no singular points and hence it is absolutely irreducible.

\begin{theorem}\label{Th:AbsolutelyIrreducible}
 Suppose that {$c$ satisfies Condition \eqref{Eq:Th}}.
 Then, $\mathcal{C}_{f,c}$ is absolutely irreducible.
\end{theorem}




We want to prove that if $q$ is large enough there exists $t_0\in \mathbb{F}_{q}$ such that the equation $(x+1)^d-cx^d=t_0$ has more than two solutions, i.e. $x^d$ is not exceptional PcN nor APcN. To this end we will investigate the geometric and the algebraic Galois groups of the polynomial $F_{c,d}(t,x)=(x+1)^d-cx^d-t$. 

More in details, consider $G^{arith}_{c,d}=\textrm{Gal}(F_{c,d}(t,x):\mathbb{F}_q(t))$ and $G^{geom}_{c,d}=\textrm{Gal}(F_{c,d}(t,x):\overline{\mathbb{F}}_q(t))$. They are both subgroups of $\mathcal{S}_d$, the symmetric group over $d$ elements. Our aim is to prove that $G^{geom}_{c,d}=\mathcal{S}_d$. This would force that $G^{geom}_{c,d}=\mathcal{S}_d=G^{arith}_{c,d}$, since $G^{geom}_{c,d}\leq G^{arith}_{c,d}$ and therefore by  Chebotarev density Theorem~\cite{kosters2014short},  one obtains the existence of  a specialization $t_0\in \mathbb{F}_q$ for which  $F_{c,d}(t_0,x)$ splits into $d$ pairwise distinct  linear factors $(x-x_i)$ defined over $\mathbb{F}_q$ and therefore $(x+1)^d-cx^d$ cannot be a permutation or 2-to-1 and $x^d$ is not PcN nor APcN.

\begin{lemma}\label{Lemma:GaloisGroup}
{Let  $c$ satisfy Condition \eqref{Eq:Th}}.
The geometric Galois group $G^{geom}_{c,d}$ coincides with $\mathcal{S}_d$.
\end{lemma}
\proof
First we prove that the geometric Galois group of  $F_{c,d}(t,x)=(x+1)^d-cx^d-t\in \mathbb{F}_q[x]$ is primitive (i.e. it does not act on a nontrivial partition of the underlying set). 
Let $M$ be the splitting field of $F_{c,d}(t,x)$ and $G$ be the Galois group of $F_{c,d}(t,x)$ over $\mathbb{F}_q(t)$. Let $x$ be a root of $F_{c,d}(t,x)$ and consider the extension $\mathbb{F}_q(x):\mathbb F_q(t)$. Clearly, $t=(x+1)^d-cx^d=f_{c,d}(x)$ by definition. As a consequence of L\"uroth's Theorem, $f$ is indecomposable (i.e. it cannot be written as a composition of two non-linear polynomials) if and only if $G$ is a primitive group; see \cite[Proposition 3.4]{Fried}. 

To this end, suppose that $f_{c,d}(x)=h_1(h_2(x))$, for some $h_1(x),h_2(x)\in \overline{\mathbb{F}_q}[x]$, with $\deg(h_1(x)),\deg(h_2(x)) \in [2,\ldots,d/2]$. Then 
$$(h_2(X)-h_2(Y)) \mid (f_{c,d}(X)-f_{c,d}(Y))=(h_1(h_2(X))-h_1(h_2(Y))).$$
By Theorem~\ref{Th:AbsolutelyIrreducible}, $\mathcal{C}_{f,c}$ is absolutely irreducible and then $h_2(X)-h_2(Y)=X-Y$, which contradicts $\deg(h_2(x))>1$. Therefore $\textrm{Gal}(F_{c,d}(t,x): \overline{\mathbb F}_q(t))$ is  primitive.

Now we prove that there exists $t_0 \in \overline{\mathbb{F}_q}$ such that $(x+1)^d-cx^d=t_0$ has exactly $d-1$ roots in $\overline{\mathbb{F}_q}$. Elements $t_0\in \overline{\mathbb{F}_q}$ for which $(x+1)^d-cx^d=t_0$ has a repeated root $x_0$ are such that 
$$(x_0+1)^{d-1}-cx_0^{d-1}=0, \qquad t_0=c x_0^{d-1}.$$
Suppose that there exists another repeated root $y_0\neq x_0$ of $(x+1)^d-cx^d=t_0$. Then 
$$
\left\{
\begin{array}{l}
(x_0+1)^{d-1}-cx_0^{d-1}=0\\ 
x_0^{d-1}=y_0^{d-1}\\
(y_0+1)^{d-1}-cy_0^{d-1}=0.
\end{array}
\right.$$
which is equivalent to \eqref{Eq:SingularPoints}. So each $t_0$ has at most one repeated root. Note that a repeated root $x_0$ is at most a double root of $(x+1)^d-cx^d=t_0$ since otherwise $(x_0+1)^{d-2}=cx_0^{d-2}$ and a contradiction easily arises from $(x_0+1)^{d-1}=cx_0^{d-1}$.  Therefore each root of $(x+1)^{d-1}-cx^{d-1}$ (they are pairwise distinct) provides a $t_0=(x_0+1)^d-cx_0^d$ such that the equation $(x+1)^d-cx^d=t_0$ has exactly $d-1$ roots in $\overline{\mathbb{F}_q}$.

Let $r$ be such that the element $t_0$ obtained above belongs to $\mathbb F_{q^r}$. This means   that $(x+1)^{d}-cx^{d}-t_0$ has exactly one factor of multiplicity 
$2$ and all the others of multiplicity $1$. Let now $M$ be the splitting field of 
$F_{c,d}(t,x)$ over $\mathbb F_{q^r}(t)$. Let $R$ be a place of $M$ lying above $t_0$.
Now, using Lemma \ref{orbits} we obtain that the decomposition group $D(R\mid t_0)$ has a cycle of order exactly $2$ and fixes all the other elements of $H=\mathrm{Hom}_{\mathbb F_q(t)}(\mathbb F_q(x),M)$, where $x$ is a root in $\overline{\mathbb F_q(t)}$ of $F_{c,d}(t,X)$ ($H$ can be simply thought as the set of roots of $F_{c,d}(t,X)$ in $\overline{\mathbb F_q(t)}$). Now pick any element $g\in D(R\mid t_0)$ that acts non-trivially on $H$. This element has to be a transposition, which in turn forces $\textrm{Gal}(F_{c,d}(t,x):\mathbb F_{q^{ru}}(t))$ to contain a transposition for any $u\in \mathbb N$ and therefore in particular that $\textrm{Gal}(F_{c,d}(t,x): \overline{\mathbb F}_q(t))$ contains a transposition.

We already know that $\textrm{Gal}(F_{c,d}(t,x): \overline{\mathbb F}_q(t))$ is  primitive. Now using Theorem~\ref{thm:hemultsn} with $s=2$ we conclude that both $\mathcal{S}_d=\textrm{Gal}(F_{c,d}(t,x): \overline{\mathbb F}_q(t))$ and $\textrm{Gal}(F_{c,d}(t,x): \mathbb F_q(t))=\mathcal S_d$.
\endproof

\begin{theorem}\label{Th:NoPcN_APcN}
{Let $c$  satisfy Condition \eqref{Eq:Th}}.
Then $x^d$ is not exceptional PcN nor APcN.
\end{theorem}
\proof
Consider  $F=\mathbb F_q(t)$ and $L=F(z)$, where $z$ is a root of $F_{c,d}(t,x)\mid \overline{\mathbb F}_q(t)$. 
Lemma \ref{Lemma:GaloisGroup} tells us that  the field of constants of the Galois closure of $L:F$ is trivial, as the geometric Galois group of $F_{c,d}(t,x)$ is equal to the arithmetic one. Let $C$ be the constant in Theorem~\ref{thm:existence_tot_split}. Using now Theorem~\ref{thm:existence_tot_split} we have that if $q>C$  there exists a specialization $t_0\in  \mathbb F_q$ such that $F_{c,d}(t,x)$ is totally split and therefore $f_{c,d}(x)=t_0$ has $d$ solutions in $\mathbb{F}_q$. The claim follows.
\endproof

Finally, we list a couple of  open problems. 
\begin{open}
Non-existence results for PN or APN functions have been obtained using a number of different methods. It would be interesting to check whether such methods apply also to PcN and APcN for $c\neq 1$.
\end{open}

\begin{open}
If $p=2$, as already mentioned, no PN functions exist.  A different  definition of planar functions was given by Zhou~\cite{Zho2013}: a function  $f:\mathbb{F}_q\to\mathbb{F}_q$ is \emph{pseudo-planar} if, for each nonzero $\epsilon\in\mathbb{F}_q$, the function
\begin{equation}\label{Eq:Planar_even_widehat}
x\mapsto \widehat {f}_{\epsilon}(x):=f(x+\epsilon)+f(x)+\epsilon x
\end{equation}
is a permutation of $\mathbb{F}_q$. As shown by Zhou~\cite{Zho2013} and Schmidt and Zhou~\cite{schmidt_planar_2014},  pseudo-planar functions have similar properties and applications as their counterparts in odd characteristic. It is natural to extend such a definition to different $c$.  We call a function $f(x)$ {\em pseudo-$PcN$} if for all $c,\epsilon \in \mathbb{F}_q$, $\epsilon \neq 0$, $$f(x+\epsilon)+c f(x)+\epsilon x$$
 is a permutation of $\mathbb{F}_q$. Can these functions have the same applications as ``normal" PcN or APcN? 
\end{open}

\section*{Acknowledgment}

The research of D. Bartoli was supported by the Italian National Group for Algebraic and Geometric Structures and their Applications (GNSAGA - INdAM). 
The research of M. Calderini was supported by Trond Mohn Foundation.






\bibliographystyle{elsarticle-num-names}



\end{document}